\documentclass[a4paper,10pt,reqno]{article}

\usepackage{a4wide, fullpage}
\usepackage[latin1]{inputenc}
\usepackage{hyperref}

\usepackage{graphicx}
\usepackage{amsmath}
\usepackage{amsthm}
\usepackage{amssymb}
\usepackage{color}
\usepackage{eepic}

\usepackage{enumerate}
\usepackage{url}

\theoremstyle{plain}
\newtheorem{theorem}{Theorem}[section]
\newtheorem{proposition}[theorem]{Proposition}
\newtheorem{lemma}[theorem]{Lemma}

\newtheorem{corollary}[theorem]{Corollary}

\newtheorem{definition}[theorem]{Definition}

\theoremstyle{definition}

\theoremstyle{remark}
\newtheorem{remark}[theorem]{Remark}

\newcommand{\N}{\ensuremath{\mathbb{N}}}

\newcommand{\R}{\ensuremath{\mathbb{R}}}
\renewcommand{\P}{\ensuremath{\mathbb{P}}}
\newcommand{\E}{\ensuremath{\mathbb{E}}}

\newcommand{\noemi}[1]{\textcolor{black}{\textrm{#1}}}
\newcommand{\noeminew}[1]{\textcolor{black}{\textrm{#1}}}

\title{\noemi{Graphical representation} of certain moment dualities and application to population models with balancing selection}
\author{Sabine Jansen\footnote{
Mathematical Institute, Leiden University,
Postbus 9512,
2300 RA Leiden, The Netherlands.
E-Mail: sabine.jansen@math.leidenuniv.nl}\hspace{.2cm}  
and Noemi Kurt\footnote{Institut f\"ur Mathematik, TU Berlin MA 7-5, Stra\ss e des 17. Juni 136, D-10623 Berlin. E-Mail: kurt@math.tu-berlin.de. Supported in part 
by DFG project BL 1105/2-1.}}
\date{\today}

\begin{document}
\maketitle

\begin{abstract}
We investigate dual mechanisms for interacting particle systems. Generalizing an approach of Alkemper and Hutzenthaler in the case of coalescing duals, we show that a simple 
linear transformation leads to a moment duality of suitably rescaled processes. More precisely, we show how dualities of interacting particle systems of the form 
$H(A,B)=q^{|A\cap B|}, A,B\subset\{0,1\}^N, q\in[-1,1),$ are rescaled to yield moment dualities of rescaled processes. We discuss in particular the case $q=-1,$ which explains 
why certain population models with balancing selection have an annihilating dual process.  We also consider different values of $q,$ and answer a question by Alkemper and 
Hutzenthaler.   \\


\end{abstract}

\section{Introduction and main result}
Dualities have proved to be a powerful tool in the analysis of interacting particle systems and population models. For interacting particle systems, one generally considers 
two kind of duals: coalescing and annihilating duals, \cite{Liggett, Griffeath, SudburyLloyd}. In connection with population models, rescaled interacting particle systems and 
their limits are of considerable interest, and it is natural to ask in which sense rescaling preserves dualities.  Alkemper and Hutzenthaler \cite{AlkemperHutzenthaler} consider 
the case of coalescing dual mechanisms, and derive a `prototype' moment duality under rescaling. Swart \cite{Swart06} uses a similar idea to obtain dualities of stepping stone 
models. In this paper, we consider a general form of a duality for interacting particle systems, cf. \cite{SudburyLloyd}. This includes coalescing as well as \emph{annihilating dual 
mechanisms}. We prove a `prototype' moment duality of linearly transformed rescaled processes, in a similar fashion as for the coalescing case, 
and we discuss the situation for annihilating duals in some more details. As an application, we consider \noemi{one-dimensional} branching annihilating processes and their duals. Our approach explains 
why population models with balancing selection generally have an annihilating dual process, as was  found, for example,
in \cite{BEM} \noemi{in a spatial population model}. We \noeminew{also introduce randomized dual mechanisms, and anser a
question posed in \cite{AlkemperHutzenthaler}. Finally we discuss connections to the Lloyd-Sudbury approach, 
\cite{SudburyLloyd, Swart06}.}

For a Markov process $(X_t)_{t\geq 0}$ we write $\P_x$ for the law of the process started in $x,$ and $\E_x$ for the corresponding expectation.
Two Markov processes $(X_t)_{t\geq 0},(Y_t)_{t\geq 0}$ with state spaces $E$ and $F,$ respectively, are called \emph{dual} with respect to the \emph{duality function} 
$H:E\times F\to \R$ if for all $t\geq 0$, $x\in E$, $y\in F$ the equality
\begin{equation}\label{eq:duality}
\E_x[H(X_t,y)]=\E_y[H(x,Y_t)]
\end{equation}
holds. This means that the long-term behaviour of one process is -- to some extent -- determined by the long-term behaviour of the other process. The usefulness of a duality
depends on the duality function $H.$ If for example $(X_t)$ takes values in $\R,$ and $(Y_t)$ in $\N,$ we call a duality with respect to the function
$$H(x,y)=x^y$$
a \emph{moment duality}, since it determines all the moments of $X_t.$ 
For practical purposes, it is often useful to have a \emph{pathwise} construction of the dual processes, which means that they can be constructed on the same probability 
space in some explicit way, for example as functions of one underlying driving process. In the case of interacting particle systems, this construction is usually provided 
by the \emph{graphical representation}, \cite{Harris, Griffeath}. We explain this below in the setup that we use for the present paper.

Let $N\in\N,$ and let $E_N:=\{0,1\}^N.$ We write $x\in E_N$ as a vector $x=(x_i)_{1\leq i\leq N}.$ A partial order on $E_N$ is given by 
$x\leq y\Leftrightarrow x_i\leq y_i\;\forall 1\leq i\leq N.$ We write $x\wedge y$ for the minimum of $x$ and $y$ with respect to this ordering. Let $(X^N_t)_{t\geq 0}$ and 
$(Y_t^N)_{t\geq 0}$ denote Markov processes defined on some probability space $(\Omega, \mathcal F,\P)$ with values in $E_N$, $X_t^N=(X_t^N(i))_{i=1,...,N}.$ Let 
$A_t^N:=\{i: X_t^N(i)=1 \}$ and $B_t^N:=\{i: Y_t^N(i)=1\}$; this defines processes taking values in the subsets of $\{1,...,N\}.$ We write $|X_t^N|:=|A_t^N|$ for 
the cardinality of the set $A_t^N,$ that is for the number of $1$'s. Sudbury and Lloyd \cite{SudburyLloyd} argue that in this context, duality functions that are 
functions of $A \cap B$ alone  should be of the form
$$H(A,B)=q^{|A\cap B|},\quad A,B\subset \{1,...,N\},$$
for some $q\in \R\setminus\{1\}.$ We take this as a motivation to say that two $E_N-$valued Markov processes $(X_t^N), (Y_t^N)$ are \emph{$q-$dual} if
\begin{equation}\label{eq:q-dual}
\E_x[q^{|X_t^N\wedge Y_0^N|}]=\E_y[q^{|X_0^N\wedge Y_t^N|}]\;\;\,\,\,\forall x,y\in E_N,\ t\geq 0.
\end{equation}
That is, the duality function is $H(x,y)=q^{|x\wedge y|}.$ Special cases are $q=0,$ which is called \emph{coalescing duality}, and $q=-1$, which is called 
\emph{annihilating duality}. In these cases  the duality function becomes $0^{|x\wedge y|}=1_{\{x\wedge y=0\}},$ and $(-1)^{|x\wedge y|}=1-2\times 1_{\{|x\wedge y| 
\mbox{ is odd}\}},$ respectively.

We now describe the graphical representation for such dualities.
For each $i\in\{1,...,N\},$ draw a 
vertical line of length $T,$ which represents time up to a finite end point $T.$ We consider ordered pairs $(i,j)$ with $i,j\in\{1,...,N\}.$ For each such pair, run $m\in\N$ 
independent Poisson processes with parameters $(\lambda_{ij}^k), k=1,...,m.$ At the time of an arrival draw an arrow from the line corresponding to $i$ to the line corresponding 
to $j,$ marked with the index $k$ of the process. Do this independently for each ordered pair, for each $k=1,...,m.$ For each $k,$ we define functions 
$f^k,g^k:\{0,1\}^2\to\{0,1\}^2.$ A Markov process $(X_t^N)$ with c{\`a}dl{\`a}g paths is then constructed by specifying an initial condition $x=(x_i)_{i=1,...,N},$ and the
following dynamics: $X_t^N=x$ until the time of the first arrow in the graphical representation. If this arrow points from $i$ to $j$ and is labelled $k,$ then the pair $(x_i,x_j)$ 
is changed to $f^k(x_i,x_j),$ and the other coordinates remain unchanged. Go on until the next arrow, and proceed exactly in the same way. The dual process $(Y_t^N)$ is 
constructed using the same Poisson processes, but started at the final time $T>0,$ running time backwards, inverting the order of all arrows, and using the functions $g^k$ 
instead of $f^k.$\\
\noemi{This kind of construction goes back to Harris \cite{Harris} and is of
widespread use. A detailed account can be found in Griffeath \cite{Griffeath} and Liggett \cite{Liggett}. In the representations considered there, 
the interpretation of the mechanisms $f^k, g^k$ is such that one thinks of a particle at the tail of an arrow in the graphical representation having some 
effect on the configuration at the tip, for example by jumping there, or by branching, and subsequent coalescence, or annihilation, or death. 
The rates of the Poisson processes then naturally have the interpretation of giving a rate \emph{per particle} for some event to happen. In our case, the functions $f^k, g^k$ are
considered to act on
\emph{pairs of sites} with a certain rate, whether or not the sites are occupied. Note that given a process, the graphical representation is of course not unique, since 
different mechanisms can be combined to have the same effect.}

Following \cite{AlkemperHutzenthaler}, we call the functions $f^k,g^k$ \emph{basic mechanisms}, and we generalize the definition of dual basic mechanisms given by Alkemper 
and Hutzenthaler. For $x=(x_1,x_2)\in\{0,1\}^2$ we use the notation $x^\dagger:=(x_2,x_1)$; the dagger accounts for the reversal of an arrow.
\begin{definition}\label{def:dual_mechanism}
Two basic mechanisms $f,g:\{0,1\}^2\to\{0,1\}^2$ are called \emph{$q$-dual mechanisms} if and only if
\begin{equation}
\label{eq:dual_mechanism}
q^{|x\wedge (g(y^\dagger))^\dagger|}=q^{|f(x)\wedge y|}\,\,\,\,\forall x,y\in\{0,1\}^2.
\end{equation}
\end{definition}
It is \noemi{easy to see, cf. Lemma \ref{lem:dual_mech_proc}}, that two processes constructed using $q-$dual mechanisms are $q-$dual processes. 
\begin{center}
\setlength{\unitlength}{1cm}
\begin{picture}(7,5)
\put(1.5,1.5){\line(0,1){3}}
\put(5.5, 1.5){\line(0,1){3}}
\put(1.4,1){$x_1$}
\put(1.4, 4.7){$y_1$}
\put(5.4,1){$x_2$}
\put(5.4, 4.7){$y_2$}
\put(1.5,3.6){\vector(1,0){4}}
\multiput(1.7, 3.4)(0.3,0){13}{\line(1,0){0.2}}
\put(1.7, 3.4){\vector(-1,0){0.2}}
\put(0,3.8){$(f(x))_1$}
\put(5.6,3.8){$(f(x))_2$}
\put(0,3.0){$(g(y^\dagger))_2$}
\put(5.6, 3.0){$(g(y^\dagger))_1$}
\put(2.7,0.3){\scshape{Figure 1}}
\end{picture}
\end{center}
\begin{lemma}\label{lem:dual_mech_proc}
Fix $m\in\N$, $q\in\R\setminus\{1\}$ and $T>0$. For every $k=1,...,m$, let $f^k, g^k$ be $q-$dual basic mechanisms. \noemi{Consider
independent Poisson processes with parameters 
$\lambda^k_{ij},$ $\mu^k_{ij}, k=1,...,m, i,j\in E_N,$ which satisfy $\lambda_{ij}^k=\mu^k_{ji}$ for all $k,i,j.$ Let $X^N_0$ and 
$Y^N_0$ be $E_N-$valued random variables, independent of one another and of the Poisson processes.}
Let $(X^N_t), (Y_t^N)$ be Markov processes with state 
space $E_N,$ initial conditions $X_0^N, Y_0^N,$ constructed using the mechanisms $f_k$ and $g_k$,
respectively \noemi{driven by the Poisson processes.} Then there exists a process $(\hat{Y}_t^N)$ such that
\begin{equation} \label{eq:dual_mech_proc}
	\hat{Y}_t^N\stackrel{d}{=}Y_t^N\,\,\,\mbox{
and }\,\,\,q^{|{X}_T\wedge \hat{Y}_0|}=q^{|{X}_{t}\wedge \hat{Y}_{T-t}|}=q^{|X_0\wedge \hat{Y}_{T}|}\,\,\,\,a.s.\,\,\forall \,0\leq t\leq T.
\end{equation}
\end{lemma}

\begin{proof}
Since we assume \noemi{ $\lambda_{ij}^k=\mu^k_{ji}$,} we can construct $\hat{Y}^N_t$ from the graphical representation of $(X_t^N),$ using the same realization 
of the Poisson processes, reversing time and the directions of all the arrows. It is clear from the construction that
then $\hat{Y}_t^N\stackrel{d}{=}Y_t^N.$ \noeminew{Assume there is an arrow from $i$ to $j$ at time $t$ in the graphical representation,
and let $x:=(X_{t-}^N(i),X_{t-}^N(j)), y:= (\hat{Y}^N_{(T-t)-}(i), \hat{Y}^N_{(T-t)-}(j)).$ Then we have
$|X^N_{t-}\wedge\hat{Y}^N_{(T-t)+}|-|x\wedge(g^k(y^\dagger))^\dagger|=|X^N_{t+}\wedge \hat{Y}^N_{(T-t)-}| - |f^k(x)\wedge y|,$
and therefore}
$q^{|{X}^N_{t-}\wedge \hat{Y}^N_{(T-t)+}|}=q^{|X^N_{t+}\wedge \hat{Y}^N_{(T-t)-}|}$ holds (see Figure 1). For some more details, in the case of 
coalescing mechanisms, compare the proof of Proposition 
2.3 of \cite{AlkemperHutzenthaler}.
\end{proof}

Taking expectations, the following is then obvious.

\begin{corollary}
In the situation of Lemma \ref{lem:dual_mech_proc}, the processes $(X_t^N)$ and $(Y_t^N)$ are $q-$dual.
\end{corollary}

\begin{remark}
We note that Lemma \ref{lem:dual_mech_proc} tells us that $(X_t^N)$ and $(Y_t^N)$ are dual in a very strong sense, namely, \noemi{for fixed $T>0,$ the equation 
$H(X_t^N,\hat{Y}^N_0)=H(X^N_0,\hat{Y}_t^N)$ holds almost surely for all $0\leq t\leq T$} instead of just in expectation. We call such processes \emph{strongly pathwise dual}. We have just 
seen that a construction via graphical representation and $q$-dual basic mechanisms automatically leads to a strong pathwise duality.
\noemi{Another example for strong pathwise duality obtained from a graphical representation is given in 
\cite{CliffordSudbury}, where stochastically monotone processes on totally ordered spaces were shown to be dual with respect to the 
duality function $1_{\{x\leq y\}}$ in a pathwise sense.}
\end{remark}

We are now ready to state and prove the main result of this article. 
\noeminew{Here, we are interested in one-dimensional processes which may be obtained by rescalings of
$|X_t^N|, |Y_t^N|$ where $X_t^N$ and $Y_t^N$ are $q-$dual finite interacting particle systems. 
Therefore, we require the particle processes to be exchangeable at all times, which means that given $|X^N_t|,$ all configurations 
$X^N_t$ with $|X^N_t|$ ones are equally likely.}

\begin{theorem}
\label{thm:prototype_general}
Let $(X^N_t)$, $(Y^N_t)$ be Markov processes with state space $E_N$ that 
are $q_N-$dual for some $q_N\in[-1,1)$.  
Choose exchangeable initial conditions $X_0^N, Y_0^N\in E_N$ \noemi{independent of one another}, fixing $|X_0^N|=k_N, |Y_0^N|=n_N$, and suppose that $X_t^N$ and $Y_t^N$ stay exchangeable for all $t>0$.
Assume that $n_N/N\to 0$ and $\E[|Y^N_{t_N}|/N]\to 0$ as $N\to \infty$, for some time scale $t_N\geq 0$.  
Then
$$\lim_{N\to\infty}\E\left[\left(1+(q_N-1)\frac{|X_0^N|}{N}\right)^{|Y^N_{t_N}|}\right]=
\lim_{N\to\infty}\E\left[\left(1+(q_N-1)\frac{|X_{t_N}^N|}{N}\right)^{|Y^N_{0}|}\right],$$
provided that the limits exist.
\end{theorem}

Theorem \ref{thm:prototype_general} applies, for example, to processes constructed from basic mechanisms with rates $\lambda_{ij}^k$
that do not depend on $i$ and $j,$ \noeminew{in that case, exchangeability of the initial conditions implies that the processes stay exchangeable 
at all times. All our later examples fall into this class. This condition is however not necessary, as can be seen by considering the lookdown construction \cite{DonnellyKurtz}.}

Depending on the scaling, Theorem \ref{thm:prototype_general} may lead to a moment duality, if $\frac{|X_t^N|}{N}\to X_t,$ and $|Y_t^N|\to Y_t,$ as we then get 
$\E[(1+(q-1) X_0)^{Y_t}]=\E[(1\noemi{+}(q-1)X_t)^{Y_0}].$ If $X^N$ and $Y^N$ have the same scaling, we may get a Laplace duality, 
that is $H(x,y)=e^{-\lambda x y}$ for some $\lambda\in\R,$ see \noeminew{Theorem 4.3 of} \cite{AlkemperHutzenthaler} for an example.

\begin{proof}
The proof relies on the simple fact that, \noemi{by independence and exchangeability,} the distribution of $|X\wedge Y|$ given $|X|$ and $|Y|$ is approximately 
binomial with parameters $|Y|$ and $\frac{|X|}{N},$
provided that $|Y|$ is small with respect to $N.$ Indeed, $|X\wedge Y|$ follows a hypergeometric distribution, since it is obtained by distributing the $|Y|$ 1's of the 
$Y$-configuration onto the $|X|$ 1's of the $X$-configuration, without hitting the same 1 twice. Approximating the hypergeometric distribution by a binomial distribution 
will give us the result. Let $Z^N\sim\mathrm{Bin}\left(n_N,\frac{x_N}{N}\right)$ with $x_N\in\{0,...,N\}$ and $n_N/N\to 0.$ By Theorem 4 of \cite{DiaconisFreedman}, we can 
bound the total variation distance between the hypergeometric and the binomial distribution as
$$\|\mbox{Hyp}(N, x_N, n_N)-\mbox{Bin}(n_N, \frac{x_N}{N})\|_{\mbox{TV}}\leq \frac{4 n_N}{N}.$$
Since we assumed $q_N\in[-1,1),$ we obtain
\begin{equation*}
\begin{split}
\E\left[q_N^{|X_{t_N}^N\wedge Y_0|}\;\big|\;|X_{t_N}^N|=x_N, |Y_0^N|=n_N\right]=& \sum_{k=0}^{n_N}q_N^k\P(|X_t^N\wedge Y_0|=k\;\mid\;|X_t^N|=x_N, |Y_0^N|=n_N)\\
=&\E\left[q_N^{Z^N}\right]+o(1),
\end{split}
\end{equation*}
where $\E[q_N^{Z^N}]$ is just the probability generating function of the binomial variable $Z^N.$ This is well known to be 
$$\E\left[q_N^{Z^N}\right]
=\left(q_N\frac{x_N}{N}+\left(1-\frac{x_N}{N}\right)\right)^{n_N}=\left(1+(q_N-1)\frac{x_N}{N}\right)^{n_N}.$$
Averaging over the initial conditions $X_0^N$ with $|X_0^N|=k_N,$ we obtain
$$\E\left[q_N^{|X_t^N\wedge Y_0^N|}\right]=\E\left[\left(1+(q_N-1)\frac{|X_{t_N}^N|}{N}\right)^{|Y_0^N|}\right]+o(1).$$
In the same way, using $\E[|Y^N_t|]/N\to 0,$ we get
$$\E\left[q_N^{|X_0^N\wedge Y_{t_N}^N|}\right]=\E\left[\left(1+(q_N-1)\frac{|X_{0}^N|}{N}\right)^{|Y_{t_N}^N|}\right]+o(1).$$
By duality, 
$$\E\left[\left(1+(q_N-1))\frac{|X_{t_N}^N|}{N}\right)^{|Y_0^N|}\right] =\E\left[\left(1+(q_N-1)\frac{|X_{0}^N|}{N}\right)^{|Y_{t_N}^N|}\right]+o(1).$$
Letting $N\to\infty$ gives the desired result.
\end{proof}

For the binomial approximation, it was necessary to assume that $Y^N/N\to 0.$ We now give a result for the case that both $X_t^N$ and $Y_t^N$ scale with $N.$ \noeminew{This leads
to a Laplace duality in a situation that was not covered in \ref{thm:prototype_general}.}

\begin{proposition}\label{prop:Laplace_scaling}
Let $(X^N_t),(Y^N_t)$ be Markov processes with state space $E_N$  that 
 are $q_N$-dual for some $q_N$ such that $\lim_{N\to\infty} N(q_N-1)=-\lambda\in(-\infty,0].$ 
Choose exchangeable initial conditions $X_0^N, Y_0^N\in E_N$ \noemi{independent of each other}
and suppose that both processes stay exchangeable at $t>0$.  Assume that the process $\frac{|Y^N_t|}{N}$ 
converges weakly to some process $\tilde{Y_t}$, that $\frac{|X^N_{t}|}{N}$ converges weakly to $\tilde{X}_t$. Then $(\tilde{X}_t)$ and $(\tilde{Y}_t)$ are dual with respect to 
$$H(x,y)=e^{-\lambda xy}.$$
\end{proposition}

\begin{proof}
We have
\begin{align*}
\E\left[q_N^{|X_t^N\wedge Y_0^N|}\right]& =\E\left[q_N^{\sum_{i=1}^NX_t^N(i)Y_0^N(i)}\right]\\
	&=\E\left[\left(1+\frac{N(q_N-1)}{N}\right)^{N\cdot\frac{1}{N}\sum_{i=1}^NX_t^N(i)Y_0^N(i)}\right]\rightarrow \E\left[e^{-\lambda\tilde{X}_t\tilde{Y}_0}\right],
\end{align*}
since by exchangeability \noeminew{and independence}, $\frac{1}{N}\sum_{i=1}^NX_t^N(i)Y_0^N(i)\rightarrow \tilde{X}_t\tilde{Y}_0$ \noemi{in distribution.}
\end{proof}
\begin{remark}
Note that for these results we only assume \emph{duality} of the processes, and not necessarily strong pathwise duality
in the sense of \eqref{eq:dual_mech_proc}. \noeminew{An example of a $q-$self-duality, which is not obtained form $q-$dual basic mechanisms,
but from $q-$self-dual \emph{randomized} mechanisms is given in the last section of this paper.}\\
\noemi{If all the approximating processes are constructed from a graphical representation
using $q-$dual mechanisms, these are strongly pathwise dual. However, our construction is not consistent, and therefore 
we do not directly give a pathwise construction of the limiting processes as is obtained from the lookdown construction 
\cite{DonnellyKurtz}. }\end{remark}

In the remainder of the paper, we discuss in some detail the case of annihilating duals and possible dual mechanisms. We restate Theorem \ref{thm:prototype_general} in this 
particular case, and discuss several examples where this result can be applied to rederive certain dualities, mostly known in the literature. The examples illuminate in particular 
the connection between annihilating duals and population models with balancing selection, as studied for example in \cite{BEM}. In the last section we consider different values of 
$q.$ The last example answers an open question of \cite{AlkemperHutzenthaler} concerning a self-duality derived in \cite{AthreyaSwart1}.

\section{Annihilating duality} \label{sec:annihilating}
\subsection{Annihilating dual mechanisms}

In this section, we discuss the special case of a $q$-duality with $q=-1,$ which is an \emph{annihilating duality}.
\noemi{Since $(-1)^{|x\wedge y|}=1-2\times 1_{\{|x\wedge y| \mbox{ is odd}\}},$} this duality relation can be written as
\begin{equation}\label{eq:ann_dual}
\P_x(|X^N_t\wedge Y_0^N|\mbox{ is odd})=\P_y(|X_0^N\wedge Y^N_t|\mbox{ is odd})
\end{equation}
for all $x,y\in E_N.$ 
In order to apply our rescaling result, we identify some basic mechanisms which lead to annihilating dualities. It is interesting to compare them to some of the coalescing 
mechanisms. In the following table, we give a list of the mechanisms that we are interested in, and afterwards discuss their duality relations. 
\begin{center}
\begin{tabular}{|c||c|c|c|c||l|}
\hline 
 & $f(0,0)$ & $f(0,1) $& $f(1,0)$ & $f(1,1)$ &   \\ 
\hline 
\hline
$f^R$ &(0,0) & (0,0) & (1,1) & (1,1) &resampling \\ 
\hline 
$f^C$ &(0,0) & (0,1) & (0,1) & (0,1) & walk-coalescence\\
\hline 
$f^A$ &(0,0) &(0,1) &(0,1)&(0,0)& walk-annihilation\\
\hline
$f^D$&(0,0) & (0,0) & (0,1) & (0,1) & death-walk\\ 
\hline 
$f^{BC}$&(0,0)&(0,1)&(1,1)&(1,1)& branching-coalescence\\
\hline
$f^{BA} $& (0,0) & (0,1) & (1,1) & (1,0) & branching-annihilation \\ 
\hline 
\end{tabular}
\end{center} 
The names given to the mechanisms are chosen to suggest an interpretation. In the resampling mechanism, the first position gives its type (0 or 1) to the second one. In the 
death-walk mechanism, a particle in the second position dies, after which a particle in the first position walks to the second position. Walk mechanisms suggest that a particle 
in the first position jumps to the second position, and either coalesces or annihilates if there is a particle present. In branching mechanisms, a particle in the first position 
produces a new particle in the second position, which either coalesces or annihilates with a particle already present.

\begin{remark}[Coalescing duals] In \cite{AlkemperHutzenthaler}, the coalescing dual mechanisms were classified (for a proof see the list of
dual mechanisms \cite{AHlist} that can be found on the homepage of M. Hutzenthaler). Concerning the dualities given in the above table, the following 
\emph{coalescing} dualities were established: 
(i) $f^R$ and $f^C$ are coalescing duals, (ii) $f^D$ is a coalescing self-dual, and (iii) $f^{BC}$ is a coalescing self-dual.
\noemi{They also show that the identity mechanism, the mechanism which maps all configurations to $(0,0)$ and the mechanism that 
maps $(0,0)\to(0,0)$ and all other configurations to $(1,1)$ are coalescing duals.}
\end{remark}


\begin{lemma}\label{lem:ann_dual}\begin{itemize}
\item[(a)] Two basic mechanisms $f, g$ are annihilating dual mechanisms if and only if
$$|x\wedge (g(y^\dagger))^\dagger|\mbox{ is odd } \Leftrightarrow |f(x)\wedge y|\mbox{ is odd.}$$
\item[(b)] With the notation of the above table, we have the following:
\begin{itemize}
\item[(i)] $f^R$ and $f^A$ are annihilating duals
\item[(ii)] $f^D$ is an annihilating self-dual
\item[(iii)] $f^{BA}$ is an annihilating self-dual.
\end{itemize}
\end{itemize}
\end{lemma}

\begin{proof}
$(a)$ is obvious. We verify $(b)$ using the table of basic mechanisms. $(i)$ We have that $f^R(x)\wedge y$ is odd if and only if $x=(1,0)$ or $x=(1,1)$, and $y=(0,1)$ or $(1,0).$ 
In both cases, $(f^A(y^\dagger))^\dagger=(1,0),$ and $(1,0)\wedge x$ is odd if and only if $x\in\{(1,0),(1,1)\}.$ By $(a)$ this proves the duality of $f^R$ and $f^A.$ For $(ii)$ 
note that $f^D(x)\wedge y$ is odd if and only if $x\in\{(1,0),(1,1)\}$ and $y\in\{(0,1),(1,1)\}.$ But then $(f^D(y^\dagger))^\dagger=(1,0),$ and the claim follows. 
$(iii)$ For $f^{BA}(x)\wedge y$ to be odd we need $x=(0,1)$ and $y\in\{(0,1),(1,1)\},$ or $x=(1,0)$ and $y\in\{(0,1),(1,0)\},$ or $x=(1,1)$ and $y\in\{(1,0), (1,1)\}.$ In all 
cases $(f^{BA}(y^\dagger))^\dagger\wedge x$ is odd, and there are no other possibilities.
\end{proof}

\begin{remark} The list of duals is not complete. For a full classification of coalescing duals see \cite{AHlist}. Note that $f^R$ and $f^D$ have both a coalescing 
and an annihilating dual mechanism. 
The death-walk-mechanism $f^D$ is $q$-self-dual for any $q\in\R:$
From the table of dual mechanisms one can check that  $|x\wedge (f^D(y^\dagger))^\dagger|=|f^D(x)\wedge  y|$ for all $x,y\in \{0,1\}^2.$
\noemi{The same is true for the identity and the mechanism that maps all configurations to $(0,0).$}
\end{remark}

\begin{remark}
It is easy to see that $q$-dual mechanisms, $q\neq 1,$ always satisfy $f(0,0)=(0,0).$ However, unlike the case of coalescing duality, a mechanism need not be monotone in order to have an 
annihilating dual, as can be seen from the self-duality of the branching-annihilating mechanism.
\end{remark}

We can now restate our Theorem \ref{thm:prototype_general} in the special case of annihilating duals. This special case is motivated by the observation, made in \cite{BEM}, that
a particular model of populations with balancing selection, after a transformation of the form $x\mapsto 1-2x,$ is dual to a double-branching annihilating process. Our result 
shows why this transformation occurs in annihilating processes. A non-spatial version of this model will be discussed as an example in the following section.

\begin{corollary}
\label{prop:prototype}
Let $(X^N_t),(Y^N_t)$ be Markov processes with state space $E_N$ such that 
\noeminew{$\P_x(|X^N_t\wedge y|\mbox{ is odd})=\P_y(|x\wedge Y^N_t|\mbox{ is odd})$}
holds for all $x,y\in E_N.$ Let $k_N,n_N\in\N,$ and choose exchangeable initial conditions $x_N, y_N\in E_N,$ \noemi{independent
of each other} such that $|x_N|=k_N, |y_N|=n_N.$ Suppose that both processes stay 
exchangeable at $t>0$, and assume that $n_N/N\to 0$ and $\E[|Y^N_{t_N}|/N]\to 0$ as $N\to\infty.$ 
Then
$$\lim_{N\to\infty}\E\left[\left(1-\frac{2|X_0^N|}{N}\right)^{|Y^N_{t_N}|}\right]=
\lim_{N\to\infty}\E\left[\left(1-\frac{2|X_{t_N}^N|}{N}\right)^{|Y^N_{0}|}\right],$$
provided that the limits exist.
\end{corollary}
As before, assuming that a limiting process $(p_t)$ of $1-\frac{2|X^N_{t_N}|}{N}$ and $n_t$ of 
$|Y_{t_N}^N|$ exists, these processes satisfy the moment duality
$$\E_n[p_0^{n_t}]=\E_p[p_t^{n_0}].$$

\begin{proof} Corollary \ref{prop:prototype} is a consequence of Theorem \ref{thm:prototype_general}, by setting $q=-1.$ It can also be understood from
the fact that the probability that a binomial random variable with parameters $n,p$ takes an odd value is given by $\frac{1}{2}\left(1-(1-2p)^n\right).$ Then we have, by 
binomial approximation,
\begin{equation*}
\begin{split}
\P(|X_{t}^N\wedge Y^N_0|\mbox{ is odd }\,|\, |X_{t}^N|=x_N, |Y_0^N|=n_N)
=\frac{1}{2}\left(1-\left(1-\frac{2x_N}{N}\right)^{n_N}\right)+o(1),
\end{split}
\end{equation*}
as in the proof of Theorem \ref{thm:prototype_general}; again duality, averaging over the exchangeable initial conditions, and taking limits, gives the result.
\end{proof}

\subsection{Examples}
In this section we derive some (mostly well-known) dualities by rescaling dualities of interacting particle systems. We will assume that the following mechanisms occur in 
the process $(X^N_t):$ 
$f^R$ occurs with rate $\frac{r_N}{N}$ for each ordered pair $(i,j)$, $i,j\in \{1,...,N\},$
$f^C$ with rate $\frac{c_N}{N},$
$f^A$ with rate $\frac{a_N}{N},$
$f^D$ with rate $\frac{d_N}{N},$
$f^{BA}$ with rate $\frac{b^a_N}{N},$
and $f^{BC}$ with rate $\frac{b^c_N}{N}.$ Moreover, set $b_N:=b_N^a+b_N^c.$\\ Consider the process $|X_t^N|$ taking values in $\{0,...,N\}.$ Note that if $|X_t^N|=k,$ then 
the number of ordered pairs of certain types is easily computed: The number of $(0,1)$-pairs (or equivalently of $(1,0)$-pairs) is equal to $k(N-k),$ the number of $(1,1)-$pairs 
is equal to $k(k-1).$
Hence, the process $|X_t^N|, t\geq 0,$ makes the following transitions:
\begin{eqnarray}
k\to k+1 &\mbox{ at rate } &\frac{r_N+b_N}{N}k(N-k),\label{eq:ktok+1}\\
k\to k-1 &\mbox{ at rate } &\frac{r_N+d_N}{N}k(N-k)+\frac{c_N+d_N+b^a_N}{N}k(k-1),\label{eq:ktok-1}\\
k\to k-2 &\mbox{ at rate } &\frac{a_N}{N}k(k-1).\label{eq:ktok-2}
\end{eqnarray}
\noemi{Note that obviously we could do with fewer mechanisms in order to define the process $|X_t^N|.$ However, playing with
the rates of the different mechansims, we can find different duals to processes constructed in this manner. 
}In the next sections, we will consider processes of this type and their duals for various values and scalings of the rates.

\subsubsection{Branching annihilating process}
Let $a_N=d_N=b_N^c=c_N=0,$ and assume $\frac{r_N}{N}\to\alpha\geq 0$ and  $b_N=b_N^a\to \beta\geq 0,$ as $N\to\infty.$ The different scaling of the mechanism is interpreted in the 
sense that in the limit, the resampling affects pairs of particles, while branching happens at a fixed rate per single particle. The rescaled discrete process $\frac{|X_t^N|}{N}$ 
has, according to \eqref{eq:ktok+1} and \eqref{eq:ktok-1}, the discrete generator 
\begin{equation*}
\begin{split}
\tilde{G}_Nf\left(\frac{k}{N}\right)=&\frac{r_N}{N}k(N-k)\left(f\left(\frac{k+1}{N}\right)+f\left(\frac{k-1}{N}\right)-2f\left(\frac{k}{N}\right)\right)\\
&+\frac{b_N}{N}k(k-1)\left(f\left(\frac{k-1}{N}\right)-f\left(\frac{k}{N}\right)\right)+
\frac{b_N}{N}k(N-k)\left(f\left(\frac{k+1}{N}\right)-f\left(\frac{k}{N}\right)\right).
\end{split}
\end{equation*}
Assume now $\frac{k}{N}\to x$ as $N\to \infty$ and $f$ twice differentiable. Then, noting
$\lim_{N\to\infty}N \left(f\left(\frac{k+1}{N}\right)-f\left(\frac{k}{N}\right)\right)=f'(x)$
and
$\lim_{N\to\infty}N^2\left(f\left(\frac{k+1}{N}\right)+f\left(\frac{k-1}{N}\right)-2f\left(\frac{k}{N}\right)\right)=f''(x),$
we see that $\tilde{G}_Nf(k/N)$ converges to
$$\tilde{G}f(x)=\beta x(1-2x)f'(x)+\alpha x(1-x)f''(x),$$
which is the generator of the one-dimensional diffusion given by the SDE
$$dX_t=\beta X_t(1-2X_t)dt+\sqrt{2\alpha X_t(1-X_t)}dB_t.$$
This is a Wright-Fisher diffusion with local drift $\beta x(1-2x)$. The drift has the effect of pushing $X_t$ towards the values 0 and $1/2$ and may be interpreted as a selection
promoting heterozygosity -- this interpretation will become more evident in the next example. Note that it is not difficult to incorporate death as well: If $d_N\to \delta>0,$  
the resulting diffusion reads
$$dX_t=\beta X_t(1-X_t)dt-\delta X_tdt+\sqrt{\alpha X_t(1-X_t)}dB_t.$$
Consider now the dual process. According to Lemma \ref{lem:ann_dual}, $(Y_t^N)$ where $f^A$ happens at rate $\frac{r_N}{N}, f^{BA}$ at $\frac{b_N}{N}$  is an annihilating dual 
of $(X_t^N).$ The  generator of $|Y_t^N|$ is 
\begin{equation*}
\begin{split}G_Nf(k)=&\frac{b_N}{N} k(N-k)\left(f(k+1)-f(k)\right)+\frac{b_N}{N} k(k-1)\left(f(k-1)-f(k)\right)\\
&+\frac{r_N}{N}k(k-1)\left(f(k-2)-f(k)\right).
\end{split} \end{equation*}
As $N\to\infty,$ when $f(n)\to0$ fast enough as $n\to \infty$, this converges to 
$$Gf(k):=\beta k\left(f(k+1)-f(k)\right)+\alpha k(k-1)\left(f(k-2)-f(k)\right),$$
which is the generator of a branching annihilating process on $\N_0$. Including death, we get
$$Gf(k):=\beta k\left(f(k+1)-f(k)\right)+\alpha k(k-1)\left(f(k-2)-f(k)\right)+\delta k \left(f(k-1)-f(k)\right).$$
\noemi{By corollary 4.8.9 of \cite{EthierKurtz} one obtains weak convergence of $(Y_t^N)$ to a process $(Y_t)$ with generator $G$ 
(noting that the compact containment condition follows from the fact that the annihilation rate is quadratic as opposed to
the linear rate of branching). It should be possible to prove by standard methods in a similar way as in 
\cite{AlkemperHutzenthaler} that $(X_t^N/N)$ converges weakly to the one-dimensional diffusion $(X_t)$ with generator 
$\tilde{G}.$}
By Corollary 
\ref{prop:prototype} we obtain for the limiting processes $(X_t), (Y_t)$ the duality
$$\E_x\left[(1-2X_t)^{Y_0}\right]=\E_y\left[(1-2X_0)^{Y_t}\right].$$
\begin{remark}
\noemi{Note that this is not a new duality. It can also be obtained from Proposition 6(b) in \cite{Swart06}, where a similar
approach is used, but relying on a slightly different type of graphical representation instead of the one we use here in terms of basic
mechanisms. It can also be obtained in the following way:}
Write $p_t:=1-2X_t$. It\^o's formula yields
$dp_t=\beta(p_t^2-p_t)dt-\sqrt{\alpha(1-p_t^2)}dB_t,$
from which -- at least heuristically -- it is easy to read off the moment duality of the process $(p_t)_{t\geq 0}$ and the 
branching annihilating process \noemi{by looking at the exponents of $p_t,$ or by a generator calculation: The generator of $(p_t)$  acts on $f(x)=x^n$ as 
$$Gf(x)=\beta(x^2-x)nx^{n-1}+\frac{\alpha}{2}(1-x^2)n(n-1)x^{n-2}=\beta n(x^{n+1}-x^n)+\alpha{n\choose 2}(x^{n-2}-x^n)$$
where the right-hand side, acting on $x^n$ as a function of $n,$ is the generator of the dual process. }Our method 
establishes this duality in a straightforward manner, and also shows why the transformation $p_t=1-2X_t$ has to be applied.
\end{remark}

\subsubsection{Double-branching annihilating process and populations with balancing selection} 
One of our motivations was to understand the transformation $x\mapsto 1-2x$ applied in \cite{BEM} in order to obtain the duality between the competing species model and 
double-branching annihilating random walk, which is parity preserving. \noemi{Note that \cite{BEM} deals with spatial models,
while our result is one-dimensional, but the connection between annihilating duality and this linear transformation is
not a spatial effect. The situation considered here} is substantially different from our last example, as a branching event produces \emph{two} 
new particles and not one, which is not taken care of by our basic mechanisms. However, it is easily implemented if we allow for multiple arrows in the graphical representation, 
or, equivalently, for basic mechanisms $f:\{0,1\}^3\to\{0,1\}^3.$ 

Assume that for each ordered pair $(i,j)$ the $f^A-$mechanism happens at rate $\frac{a_N}{N},$ and construct an additional mechanism $f$ in the following way: For each ordered 
triple $(i,j,k)$, $i,j,k=1,...,N$, draw, at rate $\frac{b_N}{N^2},$ an arrow from $i$ to $j$ and from $i$ to $k.$ Then, if an arrow is encountered, a transition $f^{BA}$ occurs 
for the two pairs $(i,j)$ and $(i,k).$ This means that at such a double transition, the state of the triple $(x_i,x_j,x_k)$ is changed according to the following table:
\begin{center}
\begin{tabular}{|c||c|c|c|c|c|c|c|c|}
\hline 
 $x$& (000)& (001)&(010) &(100)& (101)&(110)&(011)&(111) \\ 

\hline
$f(x)$ &(000) & (001) & (010) &(111)& (110) &(101)&(011)&(100) \\ 
\hline 
\end{tabular}
\end{center} 
\noemi{Note that the two $f^{BA}-$transitions commute, hence it does not matter which one is applied first.} The dual mechanism $\tilde{f}$ is given by inverting the arrows and applying the dual mechanism $f^{BA}$ to each of the two arrows, that is, to the pairs $(x_j,x_i)$ and 
$(x_k,x_i)$ with the additional rule that two $1$'s at the same place annihilate each other, that is, given by the table
\begin{center}
\begin{tabular}{|c||c|c|c|c|c|c|c|c|}
\hline 
 $x$& (000)& (001)&(010) &(100)& (101)&(110)&(011)&(111) \\ 

\hline
$\tilde{f}(x)$ &(000) & (101) & (110)&(100) & (001) &(010)&(011)&(111) \\ 
\hline 
\end{tabular}
\end{center} 
It is \noemi{easy to check} that these two mechanisms are annihilating dual mechanisms, either by direct verification, or by noting that the double-branching transition is the result of 
two $f^{BA}-$transitions happening one right after the other, cf. Figure 2 for a graphical representation \noemi{where we 
see $f^A-$transitions at time $t_2$ between sites 5 and 4 and at time $t_4$ between 3 and 2, and
$f-$transitions at time $t_1$ between 2, 1 and 3 and at $t_3$ between 2, 4 and 5.}

\begin{center}
\setlength{\unitlength}{1cm}
\begin{picture}(8,5)
\multiput(0,0.5)(2,0){5}{\line(0,1){4.5}}
\put(2,1){\vector(1,0){2}}
\put(2,1){\vector(-1,0){2}}
\put(-0.3,1){$t_1$}
\put(2,1){\circle*{0.1}}
\put(8,1.7){\vector(-1,0){2}}
\put(-0.3,1.7){$t_2$}
\put(8,1.7){\circle*{0.1}}
\put(4,4.2){\vector(-1,0){2}}
\put(4,4.2){\circle*{0.1}}
\put(-0.3,4.2){$t_4$}
\put(2,3){\line(1,0){1.9}}
\put(4.1,3){\vector(1,0){1.9}}
\put(6,3){\vector(1,0){2}}
\put(2,3){\circle*{0.1}}
\put(-0.3,3){$t_3$}
\put(4,3){\arc{0.2}{3.1415}{0}}
\put(0,0.2){1}
\put(2,0.2){2}
\put(4,0.2){3}
\put(6,0.2){4}
\put(8,0.2){5}
\put(3.3,-0.2){\scshape{Figure 2}}
\end{picture}
\end{center}

Let now $(Y_t^N)$ be the process constructed from the graphical representation, where $f^A$ happens at rate $\frac{a_N}{N},$ and $f$ at rate $\frac{b_N}{N^2}.$ Then 
$|Y_t^N|$ has the transitions
\begin{eqnarray*}
k\to k+2&\mbox{ at rate }&\frac{b_N}{N^2}k(N-k)(N-k-1),\\
k\to k-2&\mbox{ at rate }&\frac{b_N}{N^2}k(k-1)(k-2)+\frac{a_N}{N}k(k-1),
\end{eqnarray*}
since $k(N-k)(N-k-1)$ is the number of $(100)-$triples if there are $k$ 1's, etc.
Assume $\frac{a_N}{N}\to\alpha$ and $b_N\to\beta$ as $N\to\infty.$ Then
the generator of $|Y^N_t|$ converges to
$$Gf(k)=\beta k (f(k+2)-f(k))+\alpha k(k-1)(f(k-2)-f(k)),$$
which is the generator of a double-branching annihilating process. For the dual process $(X_t^N),$ with mechanisms $f^R$ and $\tilde{f},$ we obtain the transitions
\begin{eqnarray*}
k\to k+1&\mbox{ at rate }&\frac{b_N}{N^2}2k(N-k)(N-k-1)+\frac{a_N}{N}k(N-k),\\
k\to k-1&\mbox{ at rate }&\frac{b_N}{N^2}2k(k-1)(N-k)+\frac{a_N}{N}k(N-k),
\end{eqnarray*}
which yield for $N\to \infty,$ if $\frac{k}{N}\to x,$
\begin{equation*}\begin{split}\tilde{G}f(x)
=&2\beta x(1-x)(1-2x)f'(x)+\alpha x(1-x)f''(x).\end{split}\end{equation*}
$\tilde G$ is exactly the generator of the non-spatial version of the competing species model of \cite{BEM}, given by the SDE
$$dX_t=2\beta X_t(1-X_t)(1-2X_t)dt+\sqrt{2\alpha X_t(1-X_t)}dB_t.$$
For a motivation of this model as well as results on the long-term behaviour of its spatial version, see \cite{BEM}.

\section{Other values of $q$ and self-duality of the resampling-selection process}
\label{sec:resem}

\noemi{At the end of \cite{AlkemperHutzenthaler},} Alkemper and Hutzenthaler ask whether the self-duality derived in \cite{AthreyaSwart1} for the so-called resampling-selection process
\begin{equation}
\label{eq:resem-process}dX_t=\beta X_t(1-X_t)dt-\delta X_t dt+\sqrt{\alpha X_t(1-X_t)}dB_t
\end{equation}
could be \noeminew{constructed using} the approach of dual basic mechanisms \noemi{(note that in section 2.2.1 we constructed a dual, but
not self-dual process)}. There are several related questions. First, 
	is it possible to construct $(X_t)$ as the scaling limit of self-dual processes $(X_t^N)$ for interacting particle systems in such a way that the self-duality of 
	$(X_t)$ is inherited from the self-duality of $(X_t^N)$ (compare \cite{Swart06}, Prop. 6(a))? 
	Second, is it possible to explain the self-duality of the discrete process $(X_t^N)$ using a pathwise construction? Third, can we choose to construct the discrete 
	processes with $q$-dual basic mechanisms, thus obtaining interacting particle systems that are strongly pathwise dual? 
	\noeminew{An additional fourth question would be to determine whether the limiting self-duality of $(X_t)$ is still pathwise in a 
	suitable sense, a question which is not addressed in \cite{AlkemperHutzenthaler}, and which we do not address in the present paper either.}
	
	As we shall see, the answer to the first two questions is yes: There is a pathwise construction, using the basic mechanisms from Section \ref{sec:annihilating}, 
	yielding $q$-dual processes $(X_t^N)$ and $(Y_t^N)$ that rescale to resampling-selection processes. \noeminew{The answer
	to the third question, however, is no, unless we consider \emph{randomized} basic mechanisms as we will explain below.} \noemi{We start by investigating $q-$dual mechanisms.
	\begin{lemma}\label{lem:q-dual1}
	 Assume that $f$ and $g$ are $q-$dual mechanisms for some $q\notin\{-1,0,1\}.$ Then they are $q-$dual for all $q\in\R.$
	\end{lemma}
\begin{proof}
If $f$ and $g$ are $q-$dual for $q\notin\{-1,0,1\},$ then we have $|x\wedge (g(y^\dagger))^\dagger|=|f(x)\wedge y|$ for all
$x,y\in \{0,1\}^2,$ since for these $q$ the equality $q^a=q^b$ implies $a=b.$ Hence $f$ and $g$ are $q-$dual for all $q.$
\end{proof}
We also note the following:}
\noemi{\begin{lemma}\label{lem:q-monotone}
Let $f$ and $g$ be $q-$dual mechanisms for $q\notin\{-1,0,1\}.$ Then $|f(x)|\leq |x|$ and $|g(x)|\leq |x|$ for all $x\in\{0,1\}^2.$
\end{lemma}
\begin{proof}
We know that $f(0,0)=(0,0).$ Assume $f(x)=(1,1).$ Then $|f(x)\wedge (1,1)|=2,$ and by Lemma \ref{lem:q-dual1}, $|x\wedge (g(1,1))^\dagger|=2.$
But this implies $x=(1,1),$ which proves the claim.
\end{proof}}
\noeminew{\begin{remark}
Using these two lemmas and the results of \cite{AHlist}, we can give a complete classification of $q-$dual mechanisms for $q\notin\{-1,0,1\}.$ 
By lemma \ref{lem:q-dual1} we need to consider only the coalescing dual mechanisms of the table on p. 212 of 
\cite{AlkemperHutzenthaler}. Lemma \ref{lem:q-monotone} rules out mechanisms i), ii) and vi) in that table, leaving, in our notation, $f^D,$ the 
identity and the mechanism that maps all configurations to 0. Following the arguments of Alkemper and Hutzenthaler, these are the only
$q-$dual mechanisms, $q\notin \{-1,0,1\},$ modulo the transformations described in Lemma 0.3 of \cite{AHlist}.
\end{remark}}

\noeminew{Going back to our problem, Lemma \ref{lem:q-monotone} implies that any process $|X_t^N|$ where $X_t^N$ is constructed 
from $q-$dual basic mechanisms for $q\notin\{-1,0,1\}$ is
decreasing almost surely. Since this is not the case for the resem-process, the answer to the third question above is no.
However, the positive answer to the first two questions above still allows us to derive this self-duality from 
Theorem \ref{thm:prototype_general}, using an extension of the notion of dual mechanism to randomized mechanisms, 
leading to natural couplings of discrete $q$-dual processes.}\\

We construct a graphical representation of two types of arrows: One type 
occurring at rate $d,$ and the second type at rate $r+b,$ where $b=b^c.$ The first type is associated with the mechanism $f^D.$ 
The second type is associated with a random mechanism $f_q^{R,B}:\{0,1\}^2\to\{0,1\}^2$: the duality parameter $q = r/(b+r$) 
is taken as the (conditional) probability that the arrow is of the resampling type. With probability $1-q$, the arrow is of the branching-coalescence type. 
In other words, let $\zeta$ be a Bernoulli variable with parameter $q,$ then 
$$f^{R,B}_q=\zeta f^R+(1-\zeta)f^{BC}.$$
This mechanism is not self-dual to itself in the sense of definition~\ref{def:dual_mechanism}. However, it is self-dual in a weaker sense, namely if we average over exchangeable initial conditions and over $\zeta:$ 
Let $X$ and $Y$ be $\{0,1\}^2$-valued random variables that are exchangeable ($\P(X=(a,b)) = \P(X=(b,a))$) and independent of each 
other and of 
$\zeta$. We have 
\begin{equation}\label{eq:q-average}
	 \E \Bigl[q^{|f_q^{R,B}(X)\wedge Y|} \Bigr]=\E \Bigl[q^{ |X\wedge f_q^{R,B}(Y^\dagger)^\dagger|}\Bigr]. 
\end{equation} 
To see this, note that if $X=(0,0)$ or $Y=(0,0)$ or if 
$X=Y=(1,1),$ the equality 
is trivially true. If $|X|=1$ and $|Y|=1$ the equality is true by exchangeability, since all combinations 
of $(0,1)$ and $(1,0)$ for $X$ and $Y$ are equally likely. Assume $X=(1,1).$ The cases $Y=(0,1)$ and $Y=(1,0)$ are equally likely, and 
since $f_q^{R,B}((0,1))=\zeta (0,0)+(1-\zeta)(0,1),$ and $f_q^{R,B}((1,0))=(1,1),$ we obtain 
$$\E\Bigl[q^{|X\wedge f^{R,B}_q(Y^\dagger)^\dagger|}\Bigr]=\frac{1}{2}(q\cdot 1+(1-q)q+q^2)=q=\E\Bigl[q^{|f^{R,B}_q(X)\wedge Y|}\Bigr].$$
A similar identity holds, of course, if the mechanism applied to $Y$ uses a Bernoulli variable $\zeta^2$ independent of $\zeta=\zeta^1$. 

Fix a time horizon $T>0$, and $N\in \N$. For simplicity, we drop the $N$-dependence in the notation. We couple $E_N$-valued processes $(X_t)_{0\leq t\leq T}$, $(Y_t)_{0 \leq t \leq T}$ as follows: we start as in the usual graphical representation with Poisson arrows 
of two types and rates $d$ and $r+b$ for each pair of sites as explained above, and $X_0$ and $Y_0$ exchangeable $E_N$-valued random variables, 
independent of each other and of the Poisson variables. 
We add two independent sequences $(\zeta^1_k)_{k\in \N}$, $(\zeta_k^2)_{k\in \N}$ of i.i.d. Bernoulli random variables with parameter $q$, independent from $X_0,Y_0$ and $(\zeta_k)$. Almost surely, there are only finitely many arrows of the second type, occurring at moments $t_1< t_2 < \cdots$. We attach the  variable $\zeta^1_k$ to the arrow at time $t_k$, and construct $(X_t)$ by applying the corresponding mechanisms. $(Y_t)$ is constructed in a similar way, except that the arrows are used from right to left (time $T$ down to $0$) and the randomized mechanisms use the variables $\zeta_k^2$. 
The resulting processes have the property that for all $t\in (0,T)$, an analogue of Eq.~\eqref{eq:q-average}
holds for $X= X_{t-}$ and $Y = Y_{(T-t)-},$
with $\E$ the usual expectation 
or an expectation conditioned on having an arrow of the second type at time $t$.  

It seems natural to call a basic mechanism 
which satisfies \eqref{eq:q-average} a \emph{randomized $q-$self dual mechanism}. Clearly, \eqref{eq:q-average} implies $q-$duality, though 
not strong $q-$duality, of $(X_t)$ and $(Y_t),$ provided $X_t$ and $Y_t$ are exchangeable for all $t.$\\



\noemi{Fix now $\alpha,\beta,\delta>0$, set 
$a_N=c_N=b^a_N=0$, and choose rates $b^c_N\to \beta>0$, $d_N\to \delta>0,$ and $r_N/N\to \alpha/2>0$ as $N\to\infty.$ Fix $T>0$ and let $(X_t^N)_{0\leq t \leq T}$, $(Y_t^N)_{0 
\leq t \leq T}$} be processes constructed \noeminew{as above with randomized mechanisms and} rates $b^c_N$, $d_N$ and $r_N$. 
The discrete rescaled processes $(|X_{t} ^N|/N)_{0\leq t \leq T}$ and $(|Y_t^N|/N)_{0 \leq t \leq T}$ both have formal generator 
\begin{equation}\begin{split}
	G_Nf\left(\frac{k}{N}\right)=&\frac{r_N}{N}k(N-k)\left(f\left(\frac{k+1}		{N}\right)+f\left(\frac{k-1}{N}\right)-2f\left(\frac{k}{N}\right)\right)\\
&+\frac{b_N^c}{N}k(N-k)\left(f\left(\frac{k+1}{N}\right)-f\left(\frac{k}{N}\right)\right)\\
&+\frac{d_N}{N}\big(k(N-k)+k(k-1)\big)\left(f\left(\frac{k-1}{N}\right)-f\left(\frac{k}{N}\right)\right).
\end{split}\end{equation}
For $N\to\infty$ and $k/N\to x,$ this converges to
$$Gf(x)=\frac{\alpha}{2} x(1-x)f''(x)+\beta x(1-x)f'(x)-\delta xf'(x),$$
which is the generator of the diffusion \eqref{eq:resem-process}, and one can show that the rescaled processes converge to two
resampling-selection processes $(X_t)$ and 
$(Y_t)$. Now, by \eqref{eq:q-average} (or by Corollary \ref{prop:cond_q-duality} below), $(X_t^N)$ and $(Y_t^N)$ are dual 
with respect to $q_N=r_N/(b_N^c+r_N)$, and by our assumptions on the rates, we have 
$N(q_N-1)\to -2\beta/\alpha$. 
Proposition \ref{prop:Laplace_scaling} therefore yields
$$\E_x\left[e^{-(2\beta/\alpha) X_t y}\right]=\E_y\left[e^{-(2\beta/\alpha) x Y_t}\right],$$
which is the self-duality of the resampling-selection process proven in \cite{AthreyaSwart1}. 
Thus we have provided a pathwise construction of \noeminew{the self-duality of the discrete approximating processes},
and we have shown that the \noeminew{limiting} self-duality can be obtained by rescaling dual interacting particle systems. 

We should note that the latter fact was shown by Swart \cite{Swart06}.
His argument, however, starts from independent discrete processes $(X_t^N)$ and $(Y_t^N)$, and applies a duality criterion by
Sudbury and Lloyd, see the 
proof of Proposition 6 in \cite{Swart06}. In contrast, our pathwise construction using randomized mechanisms 
provides a non-trivial coupling of
underlying discrete processes, which might be of 
interest in some contexts.
\noeminew{To conclude, we mention that instead of using randomized mechanisms, we could also apply criterion derived in \cite{SudburyLloyd}, 
used by \cite{Swart06}, which easily translates into the setting of $q-$dual mechanisms.}
Sudbury and Lloyd
consider Markov processes on $\{0,1\}^\Lambda,$ for some graph $\Lambda,$ whose generator is of the form 
\begin{equation} \label{eq:sudbury-lloyd-form}
\begin{split}
\overline{G}f(x)=\sum_{i\neq j}\overline{q}(i,j)&\Big(\frac{\overline{a}}{2}x(i)x(j)\big(f(x-\delta_i-\delta_j)-f(x)\big)
+\overline{b} x(i)(1-x(j))\big(f(x+\delta_j)-f(x)\big)\\
&+\overline{c} x(i)x(j)\big(f(x-\delta_i)-f(x)\big)
+\overline{d}x(i)(1-x(j))\big(f(x-\delta_i)-f(x)\big)\\
&+\overline{e}x(i)(1-x(j))\big(f(x-\delta_i+\delta_j)-f(x)\big)\Big),\quad x\in\{0,1\}^\Lambda\end{split}
\end{equation}
with non-negative parameters $\overline{a},\ldots,\overline{e}$, and
$\overline{q}(i,j)$  defined as follows. 
When $i$ and $j$ are neighbors in $\Lambda$ (meaning that they are connected by an edge in the graph), then $\overline{q}(i,j)=1/N_i$, with $N_i$ the number of neighbors of 
$i$; when $i$ and $j$ are not neighbors, $\overline{q}(i,j)=0$. Thus when $\Lambda$ is the complete graph on $\{0,1,\ldots,N\}$,  $\overline{q}(i,j)=1/N$ for all $i\neq j$. 
The letters $\overline{a},\overline{b},\overline{c},\overline{d}$ and $\overline{e}$ refer to
annihilation, branching, coalescence, death and exclusion. \noeminew{Given the process, these rates are unique.}

We are interested in Markov processes with state space \noemi{$\{0,1\}^N$} constructed from the basic mechanisms of Section \ref{sec:annihilating} and rates chosen as 
follows: For every pair $(i,j)$, the mechanisms $f^A$, $f^{BA}$, $f^{BC}$, $f^C$, $f^D$ and $f^R$ happen at rates $a/N$, $b^a/N$, $b^c/N$, $c/N$, $d/N$ and $r/N$. The 
infinitesimal generator of this process is of the Sudbury-Lloyd form \eqref{eq:sudbury-lloyd-form} with 
\begin{equation}\label{eq:graph-to-sl-res}
	\overline{a}=2a, \quad  \overline{b}=b^a+b^c+r,\quad \overline{c}=b^a+c+d,\quad 
	\overline{d}=d+r,\quad  \overline{e}=a+c+d
\end{equation}
and $\overline{q}(i,j)=1/N$ for all $i \neq j$. We shall refer to this process as the process obtained from the basic mechanisms via the rate parameters $a$, $b^a$, $b^c$, $d$ 
and $r$. 
Note that not every Sudbury-Lloyd process can be constructed with our basic mechanisms:  for example, if $2 \overline{e}<\overline{a}$, any solution of Eq. 
\eqref{eq:graph-to-sl-res}  has negative rate parameters $c<0$ or $d<0$. Furthermore, the construction is not unique -- \noeminew{note that
\eqref{eq:graph-to-sl-res} fixes $a=\overline{a}/2$ and $b^a=\overline{c}-\overline{e}+\overline{a}/2,$ but leaves one degree of freedom in the choice
of $b^c, c, d$ and $r.$ }

Sudbury and Lloyd give several conditions for $q$-duality of their models. A concise formula is  \cite[Eq. (9)]{Sudbury}, which in our notation reads 
\begin{equation} \label{eq:sudbury}
	\overline{a}'=\overline{a}+2q\gamma,\quad \overline{b}'=\overline{b}+\gamma,\quad \overline{c}'=\overline{c}-(1+q)\gamma,\quad \overline{d}'=\overline{d} +\gamma,
	\quad \overline{e}'=\overline{e}-\gamma
\end{equation}
where $\gamma= (\overline{a}+\overline{c}-\overline{d}+ \overline{b}q)/(1-q)$. Eq. \eqref{eq:sudbury} is easily translated into a criterion for processes obtained from our 
basic mechanisms. \noeminew{This gives a necessary and sufficient condition for duality of Sudbury-Lloyd processes, see \cite{Sudbury}. Plugging 
\eqref{eq:graph-to-sl-res} into \eqref{eq:sudbury} then easily leads to the following criterion for $q-$duality of processes constructed
from basic mechanisms:}

\begin{corollary}\label{prop:cond_q-duality}
Let $(X_t), (Y_t)$ be the Sudbury-Lloyd processes obtained from our basic mechanisms with respective rate parameters $a, b^a, b^c, c, d,r$ and $a',{b^a}',{ b^c}', c', d',r'$. 
Then
\begin{enumerate} 
	\item[(a)]$(X_t)$ and $(Y_t)$ are dual with parameter $q\in\R\setminus\{1\}$ if and only if 
 \noeminew{ \begin{equation*} 
 a'=a+q\gamma ,\quad   {b^a}'=b^a,\quad {b^c}'+r'=b^c+r +\gamma,\quad c'+d'=c+d-(1+q)\gamma,\quad d'+r'=d+r+\gamma,
 \end{equation*}}
where $\gamma=(2a+(1+q)b^a +qb^c+ c- (1-q)r)/(1-q).$ 
	\item[(b)] $(X_t)$ is self-dual with parameter $q$ if and only if $q = (r- 2a-b^a-c)/(b^a+b^c+r)$. 
\end{enumerate}
\end{corollary}

\subsubsection*{Acknowledgements} Both authors wish to thank the Hausdorff Research Institute for Mathematics in Bonn for hospitality, and Matthias Hammer for many useful comments. We 
also thank two anonymous referees for many helpful and detailed suggestions which considerably improved the paper.

\end{document}